\documentclass{gtpart}     
\usepackage{pinlabel}  
\usepackage{graphicx}  

\usepackage{amsmath}
\usepackage{amssymb}
\usepackage{amsthm}

\title{Group trisections and smooth $4$--manifolds}

\author{Aaron Abrams}
\givenname{Aaron}
\surname{Abrams}
\address{Mathematics Department\\
Washington and Lee University\\\newline
Lexington VA 24450 }
\email{abramsa@wlu.edu}

\author{David T Gay}
\givenname{David T}
\surname{Gay}
\address{Euclid Lab\\ 160 Milledge Terrace\\ Athens, GA 30606\\\newline Department of Mathematics\\ University
  of Georgia\\ Athens, GA 30602}
\email{d.gay@euclidlab.org}

\author{Robion Kirby}
\givenname{Robion}
\surname{Kirby}
\address{Department of Mathematics\\ University of California\\ Berkeley, CA 94720}
\email{kirby@math.berkeley.edu}


\volumenumber{}
\issuenumber{}
\publicationyear{}
\papernumber{}
\startpage{}
\endpage{}
\doi{}
\MR{}
\Zbl{}
\received{}
\revised{}
\accepted{}
\published{}
\publishedonline{}
\proposed{}
\seconded{}
\corresponding{}
\editor{}
\version{}

\newtheorem{theorem}{Theorem}

\newtheorem{corollary}[theorem]{Corollary}

\theoremstyle{definition}
\newtheorem{definition}[theorem]{Definition}

\def\Z{\mathbb Z}

\def\cG{\mathcal{G}}
\def\cM{\mathcal{M}}

\newcommand{\into}{\ensuremath{\hookrightarrow}}

\begin{document}

\begin{abstract}    
A trisection of a smooth, closed, oriented $4$--manifold is a decomposition into three $4$--dimensional $1$--handlebodies meeting pairwise in $3$--dimensional $1$--handlebodies, with triple intersection a closed surface. The fundamental groups of the surface, the $3$--dimensional handlebodies, the $4$--dimensional handlebodies, and the closed $4$--manifold, with homomorphisms between them induced by inclusion, form a commutative diagram of epimorphisms, which we call a trisection of the $4$--manifold group. A trisected $4$--manifold thus gives a trisected group; here we show that every trisected group uniquely determines a trisected $4$--manifold. Together with Gay and Kirby's existence and uniqueness theorem for $4$--manifold trisections, this gives a bijection from group trisections modulo isomorphism and a certain stabilization operation to smooth, closed, connected, oriented $4$--manifolds modulo diffeomorphism. As a consequence, smooth $4$--manifold topology is, in principle, entirely group theoretic. For example, the smooth $4$--dimensional Poincar\'{e} conjecture can be reformulated as a purely group theoretic statement. \footnote{This work was supported by NSF grant DMS-1207721 and by two grants from the Simons Foundation (\#359873 to David Gay and \#281189 to Aaron Abrams).}
\end{abstract}

\maketitle


Let $g$ and $k$ be integers with $g \geq k \geq 0$. We fix the following groups, described explicitly by presentations:
\begin{itemize}
 \item $S_0 = \{1\}$ and, for $g>0$, $S_g = \langle a_1,b_1, \ldots, a_g, b_g \mid [a_1,b_1] \ldots [a_g,b_g] \rangle$, i.e. the standard genus $g$ surface group with standard labelled generators. We identify this in the obvious way with $\pi_1(\#^g S^1 \times S^1,*)$.
 \item $H_0 = \{1\}$ and, for $g>0$, $H_g = \langle x_1, \ldots, x_g \rangle$, i.e. a free group of rank $g$ with $g$ labelled generators. We identify this in the obvious way with $\pi_1(\natural^g S^1 \times B^2,*)$. Note that, if $g < g'$, then $H_g \subset H_{g'}$.
 \item $Z_0 = \{1\}$ and, for $k>0$, $Z_k = \langle z_1, \ldots, z_k \rangle$, i.e. a free group of rank $k$ with $k$ labelled generators. We identify this in the obvious way with $\pi_1(\natural^k S^1 \times B^3,*)$. Again, if $k < k'$ then $Z_k \subset Z_{k'}$.
\end{itemize}

Let $V$ denote the set of vertices of a cube, and let $E$ denote the set of edges.

\begin{definition}
 A $(g,k)$--trisection of a group $G$ is a commutative cube of groups as shown below, such that each homomorphism is surjective and each face is a pushout.

\begin{equation*}
\includegraphics{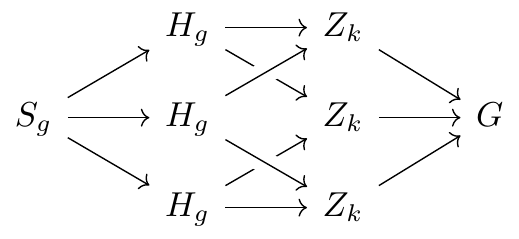}
\end{equation*}

 We label the groups $\{G_v \mid v \in V\}$ and the maps $\{f_e \mid e \in E\}$, so that a trisection of $G$ is the pair $(\{G_v\},\{f_e\})$. An isomorphism from a trisection $(\{G_v\},\{f_e\})$ of $G$ to a trisection $(\{G'_v\},\{f'_e\})$ of $G'$ is a collection of isomorphisms $h_v : G_{v} \to G'_{v}$, for all $v \in V$, commuting with the $f_e$'s and $f'_e$'s.
\end{definition}

There is a unique $(0,0)$--trisection of the trivial group.
Figure~\ref{F:31TrivTrisection} illustrates a $(3,1)$--trisection of the trivial group, which we will call ``the standard trivial $(3,1)$--trisection.'' Figure~\ref{F:31OnSurface} illustrates the same diagram more topologically. (For trisections with $g=1$ and $g=2$, see the basic $4$--manifold trisection examples in~\cite{GayKirby}; in fact, the $4$--dimensional uniqueness results in~\cite{MeierZupan}, together with Theorem~\ref{T:GroupsToManifolds} below, give uniqueness statements for group trisections with $g \leq 2$.)

\begin{figure}[h!]
\centering
\includegraphics{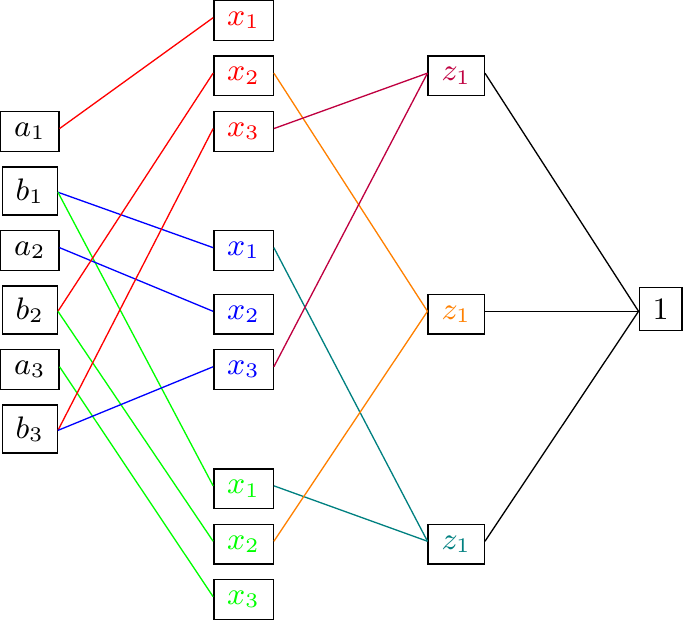}

\caption{\label{F:31TrivTrisection} A $(3,1)$--trisection of the trivial group. All maps send generators to generators or to $1$; the diagram 
shows where each map sends each generator, with the understanding that generators not shown to be mapped anywhere are mapped to $1$. 
}
\end{figure}  

\begin{figure}[h!]
 \labellist
 \tiny\hair 2pt
 \pinlabel $x_1$ [b] at 24 31
 \pinlabel $x_2$ [b] at 60 31
 \pinlabel $x_3$ [b] at 96 31
 \pinlabel $a_1$ [t] at 26 3
 \pinlabel $b_1$ [r] at 12 16
 \pinlabel $a_2$ [t] at 62 3
 \pinlabel $b_2$ [r] at 76 16
 \pinlabel $a_3$ [t] at 97 3
 \pinlabel $b_3$ [l] at 110 16
 \endlabellist
 \centering
 \includegraphics[width=.4\textwidth]{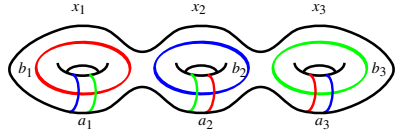}
\caption{\label{F:31OnSurface} 
The trivial $(3,1)$--trisection illustrated topologically; each color describes a handlebody filling of the genus $3$ surface, so that the curves specify the kernels of the homomorphisms.
}
 
\end{figure}

\begin{definition}
 Given a $(g,k)$--trisection $(\{G_v\},\{f_e\})$ of $G$ and a $(g',k')$--trisection $(\{G'_v\},\{f'_e\})$ of $G'$, there is a natural ``connected sum'' $(g''=g+g',k''=k+k')$--trisection $(\{G''_v\},\{f''_e\})$ of $G'' = G * G'$ defined by first shifting all the indices of the generators for the $G'_v$'s by either $g$ (when $G'_v = S_{g'}$ or $G'_v = H_{g'}$) or $k$ (when $G'_v = Z_{k'}$) and then, for each generator $y$ of $G''_v$, declaring $f''_e(y)$ to be either $f_e(y)$ or $f'_e(y)$ according to whether $y$ is in $G_v$ or $G'_v$. 
\end{definition}

\begin{definition}
 The stabilization of a group trisection is the connected sum of the given trisection with the standard trivial $(3,1)$--trisection. Thus the stabilization of a $(g,k)$--trisection of $G$ is a $(g+3,k+1)$--trisection of the same group $G = G*\{1\}$.
\end{definition}

\begin{definition}~\cite{GayKirby}
 A $(g,k)$--trisection of a smooth, closed, oriented, connected $4$--manifold $X$ is a decomposition $X=X_1 \cup X_2 \cup X_3$ such that:
 \begin{itemize}
  \item Each $X_i$ is diffeomorphic to $\natural^k S^1 \times B^3$.
  \item Each $X_i \cap X_j$, with $i \neq j$, is diffeomorphic to $\natural^g S^1 \times B^2$.
  \item $X_1 \cap X_2 \cap X_3$ is diffeomorphic to $\#^g S^1 \times S^1 = \Sigma_g$.
 \end{itemize}
 If $X$ is equipped with a base point $p$, a based trisection of $(X,p)$ is a trisection with $p \in X_1 \cap X_2 \cap X_3$. A parametrized based trisection of $(X,p)$ is a based trisection equipped with fixed diffeomorphisms (the ``parametrizations'') from the $(X_i,p)$'s to $(\natural^k S^1 \times B^3,*)$, from the $(X_i \cap X_j,p)$'s to $(\natural^g S^1 \times B^2,*)$ and from $X_1 \cap X_2 \cap X_3$ to $(\#^g S^1 \times S^1 = \Sigma_g, *)$, where $*$ in each case indicates a standard fixed base point, respected by the standard inclusions $(\#^g S^1 \times S^1 = \Sigma_g, *) \into (\natural^g S^1 \times B^2,*) \into (\natural^k S^1 \times B^3,*)$.
\end{definition}

Henceforth all manifolds are smooth, oriented and connected. Until further notice, trisected $4$--manifolds are closed.

There is an obvious map from the set of parametrized based trisected $4$--manifolds to the set of trisected groups, which we will call $\cG$; the groups are the fundamental groups of the $X_i$'s and their intersections, after identification with standard models via the parametrizations, and the maps are those induced by inclusions composed with parametrizations.

The main result of this paper is that $\cG$ induces a bijection between trisected $4$--manifolds up to trisected diffeomorphism and trisected groups up to trisected isomorphism, and that this bijection respects stabilizations
in both categories. (A {\em trisected diffeomorphism} is simply a diffeomorphism respecting the decomposition.)

\begin{theorem} \label{T:GroupsToManifolds}
 There exists a map $\cM$ from the set of trisected groups to the set of (based, parametrized) trisected $4$--manifolds such that $\cM \circ \cG$ is the identity up to trisected diffeomorphism and $\cG \circ \cM$ is the identity up to trisected isomorphism. The unique $(0,0)$--trisection of $\{1\}$ maps to the unique $(0,0)$--trisection of $S^4$, the standard $(3,1)$--trisection of $\{1\}$ maps to the standard $(3,1)$--trisection of $S^4$, and connected sums of group trisections map to connected sums of $4$--manifold trisections. Thus $\cM$ induces a bijection between the set of trisected groups modulo isomorphism and stabilization and the set of smooth, closed, connected, oriented $4$--manifolds modulo orientation preserving diffeomorphism.
\end{theorem}

Though it might not be obvious from a purely group-theoretic point of view, it follows from \cite{GayKirby} that every
finitely presented group admits a trisection, because every finitely presented group is the fundamental group of a closed, orientable $4$--manifold. Even more striking, perhaps, is that by Theorem \ref{T:GroupsToManifolds}
the collection of trisections of any particular group contains all the complexity of smooth $4$--manifolds with
the given fundamental group, including not just their homotopy types but also their diffeomorphism types.
In particular there is a subset 
of the trisections of the trivial group corresponding to the countably many exotic smooth structures on a given simply connected topological $4$--manifold, e.g. the K3 surface. (To get the full countable collection, it seems likely that $g$ must be unbounded.) An interesting problem is to understand the equivalence relation on group trisections that corresponds to {\em homeomorphisms} between $4$--manifolds.

Considering homotopy $4$--spheres, we have:

\begin{corollary}
 The smooth $4$--dimensional Poincar\'{e} conjecture is equivalent to the following statement: ``Every $(3k,k)$--trisection of the trivial group is stably equivalent to the trivial trisection of the trivial group.''
\end{corollary}

\begin{proof}
 A $(3k,k)$--trisection of the trivial group gives a $(3k,k)$--trisection of a simply connected $4$--manifold. The Euler characteristic of a $(g,k)$--trisected $4$--manifold is $2-g+3k$, so in this case we have an Euler characteristic $2$ simply connected $4$--manifold, i.e. a homotopy $S^4$.
\end{proof}

One approach to proving the Poincar\'e conjecture would be to prove first that there is a unique $(3,1)$--trisection of $\{1\}$, or at least that every $(3,1)$--trisection of $\{1\}$ gives a $4$--manifold diffeomorphic to $S^4$, and then prove that, for any $(3k,k)$--trisection of $\{1\}$, there is a nontrivial group element in the intersection of the kernels of the three maps $S_g \to H_g$ which can be represented as an embedded curve in the corresponding surface $\Sigma_g$. This would give an inductive proof since such an embedded curve would give us a way to decompose the given trisection as a connected sum of lower genus trisections. In fact, this would prove more than the Poincar\'{e} conjecture; it would also prove a $4$--dimensional analog of Waldhausen's theorem~\cite{Waldhausen}, to the effect that every trisection of $S^4$ is a stabilization of the trivial trisection and thus that any two trisections of $S^4$ of the same genus are isotopic. This strategy would be the exact $4$--dimensional parallel to the strategy outlined in~\cite{StallingsHowNot} for proving (or failing to prove) the $3$--dimensional Poincar\'{e} conjecture.

\begin{proof}[Proof of Theorem~\ref{T:GroupsToManifolds}]
 Given a $(g,k)$--trisection $(\{G_v\}, \{f_e\})$ of $G$, we will construct $\cM(\{G_v\}, \{f_e\})$ beginning with $\Sigma_g = \#^g S^1 \times S^1$. For each of the three maps $f_e : S_g \to H_g$, because these are epimorphisms it is a standard fact that there is a diffeomorphism $\phi_e: \Sigma_g \to \partial \natural^g S^1 \times B^2$ such that $\imath \circ \phi_e : \Sigma_g \hookrightarrow \natural^g S^1 \times B^2$ induces $f_e$ on $\pi_1$. See~\cite{LeiningerReid} for a proof; the sketch of the proof is as follows: Note that there is a map, well defined up to homotopy by $f_e$, from $\Sigma_g$ to a wedge of $g$ circles.  Make this transverse to one point of each circle, not the base point.  Then the inverse image of those points is a collection of embedded circles in $\Sigma_g$.  Add a $2$--handle to each circle, and then the new boundary is a collection of $2$--spheres.  Fill in each with $3$--balls resulting in a handlebody.

 Each $\phi_e$ is unique up to isotopy by Dehn's Lemma~\cite{Papakyriakopoulos}. Use these three diffeomorphisms to attach three copies of $\natural^g S^1 \times B^2$, crossed with $I$, to $\partial \Sigma_g \times D^2$ in the standard way, giving a $4$--manifold with three boundary components, each presented with a genus $g$ Heegaard splitting. (Note that the cyclic ordering of the three handlebodies is essential to determine the orientation of the resulting $4$--manifold, and that this is reflected in our definition of group trisection by the fact that the maps and groups are explicitly labelled by edges and vertices of a standard cube.)
 
 Because each pushout from the initial three maps gives a free group of rank $k$, we know that the three boundary components mentioned above are closed $3$--manifolds with rank $k$ free fundamental groups.  It is another well-known fact that each of these $3$--manifolds is diffeomorphic to $\#^k S^1 \times S^2$.  This follows from Kneser's conjecture
 (proved by Stallings~\cite{StallingsThesis}) that a free product decomposition of the fundamental group of a $3$--manifold 
 corresponds to a connected sum decomposition of the manifold, as well as Perelman's proof \cite{MorganTian}
 of the $3$--dimensional Poincar\'e conjecture that shows that no connected summand has trivial fundamental group.  
 A prime connected summand (i.e., one that doesn't decompose further) therefore has fundamental group $\Z$, and
a standard argument using the Sphere Theorem \cite{Papakyriakopoulos} and the Hurewicz and Whitehead theorems 
shows that an orientable prime 3-manifold with fundamental group $\Z$ must be $S^1\times S^2$.

 Any two ways of fillling in a connected sum of $S^1 \times S^2$'s with a $4$--dimensional $1$--handlebody differ by a diffeomorphism of the connected sum, and Laudenbach and Poenaru~\cite{LaudenbachPoenaru} proved  that any such diffeomorphism extends to a diffeomorphism of the handlebody.  So $X = \cM(\{G_v\}, \{f_e\})$, the result of filling in each $3$--manifold boundary component with a $4$--dimensional $1$--handlebody, is unique up to diffeomorphism.

 The main result of~\cite{GayKirby} is that every smooth, closed, connected, oriented $4$--manifold has a trisection, and that any two trisections of the same $4$--manifold become isotopic after performing some number of connected sums with the standard $(3,1)$--trisection of $S^4$. The connected sum operation and the $(3,1)$--trisection on the group side are constructed exactly to correspond to stabilization of manifolds via the map $\cM$.
\end{proof}

%
%
%
\bibliographystyle{plain}
%

%

\bibliography{GroupTrisections}

\begin{thebibliography}{1}

\bibitem{GayKirby}
David~T. Gay and Robion Kirby.
\newblock Trisecting 4-manifolds.
\newblock {\em Geom. Topol.}
\newblock to appear.

\bibitem{LaudenbachPoenaru}
Fran{\c{c}}ois Laudenbach and Valentin Po{\'e}naru.
\newblock A note on {$4$}-dimensional handlebodies.
\newblock {\em Bull. Soc. Math. France}, 100:337--344, 1972.

\bibitem{LeiningerReid}
Christopher~J. Leininger and Alan~W. Reid.
\newblock The co-rank conjecture for 3-manifold groups.
\newblock {\em Algebr. Geom. Topol.}, 2:37--50 (electronic), 2002.

\bibitem{MeierZupan}
J.~{Meier} and A.~{Zupan}.
\newblock {Genus two trisections are standard}.
\newblock {\em ArXiv e-prints}, October 2014.

\bibitem{MorganTian}
John Morgan and Gang Tian.
\newblock {\em The geometrization conjecture}, volume~5 of {\em Clay
  Mathematics Monographs}.
\newblock American Mathematical Society, Providence, RI; Clay Mathematics
  Institute, Cambridge, MA, 2014.

\bibitem{Papakyriakopoulos}
C.~D. Papakyriakopoulos.
\newblock On {D}ehn's lemma and the asphericity of knots.
\newblock {\em Ann. of Math. (2)}, 66:1--26, 1957.

\bibitem{StallingsHowNot}
John Stallings.
\newblock How not to prove the {P}oincar\'e conjecture.
\newblock In {\em Topology {S}eminar, {W}isconsin, 1965}, volume~60 of {\em
  Ann. of Math. Stud.}, pages 83--88. Princeton Univ. Press, Princeton, NJ,
  1966.

\bibitem{StallingsThesis}
John~R. Stallings.
\newblock {\em Some topological proofs and extensions of Gru\v{s}ko's Theorem}.
\newblock PhD thesis, Princeton University, 1959.

\bibitem{Waldhausen}
Friedhelm Waldhausen.
\newblock Heegaard-{Z}erlegungen der {$3$}-{S}ph\"are.
\newblock {\em Topology}, 7:195--203, 1968.

\end{thebibliography}

\end{document}